\documentclass[11pt,reqno]{amsart}

\usepackage{amsthm}
\usepackage{amscd}
\usepackage{amsfonts}
\usepackage{amssymb}
\usepackage{amsgen}
\usepackage{amsmath}
\usepackage{amsopn}
\usepackage{verbatim}
\usepackage{xypic}
\usepackage{xspace}
\usepackage{multicol}
\usepackage{url}
\usepackage{upref}

\theoremstyle{plain}
\newtheorem{thm}{Theorem}%[section]

\newtheorem{cor}[thm]{Corollary}

\theoremstyle{definition}

\theoremstyle{remark}

 \font\cyr=wncyr10
 \newcommand{\nc}{\newcommand}

\nc{\per}[1]{\underset{#1}{\boldsymbol \pi}\,}

 \nc{\MT}{{\rm MT}}

 \nc{\wht}{{\widehat}}
 \nc{\bwg}{{\bigwedge}}
 \nc{\mmu}{{\boldsymbol{\mu}}}
 \nc{\mal}{{{\scriptstyle \maltese}}}
 \nc{\fA}{{\mathfrak A}}
 \nc{\HH}{{\mathfrak H}}
 \nc{\ra}{\rightarrow}
 \nc{\ors}{{\vec s\,}}
 \nc{\os}{{\overset}}
 \nc{\G}{{\mathbb G}}
 \nc{\Z}{{\mathbb Z}}
 \nc{\R}{{\mathbb R}}
 \nc{\N}{{\mathbb N}}
 \nc{\ZN}{{\mathbb Z_{\ge 0}}}
 \nc{\Q}{{\mathbb Q}}
 \nc{\C}{{\mathbb C}}
 \nc{\Cnn}{{\mathbb C}_{\ge 0}}
 \nc{\Cp}{{\mathbb C}_{>0}}
 \nc{\MPV}{{\mathcal{MPV}}}
 
 \nc{\tB}{{\tilde B}}
 \nc{\Li}{{\rm Li}}
 \nc{\suf}{{\ast\,}}
 \nc{\sufq}{{\ast_q\,}}
 \nc{\gam}{{\gamma}}
 \nc{\gG}{{\Gamma}}
 \nc{\om}{{\omega}}
 \nc{\vep}{{\varepsilon}}
 \nc{\ga}{{\alpha}}
 \nc{\gl}{{\lambda}}
 \nc{\gb}{{\beta}}
 \nc{\gd}{{\delta}}
 \nc{\gs}{{\sigma}}
 \nc{\gS}{{\Sigma}}

 \nc{\gk}{{\kappa}}
  \nc{\gz}{{\zeta}}
 \nc{\tgz}{{\tilde{\zeta}}}
 \nc{\gO}{{\Omega}}
 \nc{\sif}{{\mathcal S}}
 \nc{\gt}{{\tau}}
 \nc{\Lra}{\Longrightarrow}
 \nc{\lra}{\longrightarrow}
 \nc{\fS}{{\mathfrak S}}
 \nc{\DD}{{\mathfrak D}}

 \nc{\Llra}{\Longleftrightarrow}
 \nc{\ol}{\overline}
 \nc{\lms}{\longmapsto}
 \nc{\cv}{{{\mathsf c}{\mathsf v}}}
 \nc{\zq}{{\zeta_q}}
 \nc\qup{{q\uparrow 1}}
 \nc{\us}{\underset}
 \nc{\tn}{{\tilde{n}}}
 \nc{\gD}{{\Delta}}
 \nc{\bi}{{\bf i}}
 \nc{\bfone}{{\bf 1}}
 \nc{\bfa}{{\bf a}}
 \nc{\bfb}{{\bf b}}
 \nc{\bfc}{{\bf c}}
 \nc{\bfd}{{\bf d}}
 \nc{\bfe}{{\bf e}}
 \nc{\bff}{{\bf f}}
 \nc{\bfg}{{\bf g}}
 \nc{\bfi}{{\bf i}}
 \nc{\bfj}{{\bf j}}
 \nc{\bfn}{{\bf n}}
 \nc{\bfl}{{\bf l}}
 \nc{\bfk}{{\bf k}}
 \nc{\bfm}{{\bf m}}
 \nc{\bfo}{{\bf o}}
 \nc{\bfp}{{\bf p}}
 \nc{\bfq}{{\bf q}}
 \nc{\bfr}{{\bf r}}
 \nc{\bfs}{{\bf s}}
 \nc{\bft}{{\bf t}}
 \nc{\bfu}{{\bf u}}
 \nc{\bfv}{{\bf v}}
 \nc{\bfw}{{\bf w}}
 \nc{\bfx}{{\bf x}}
 \nc{\bfB}{{\bf B}}
 \nc{\bfP}{{\bf P}}
 \nc{\bfQ}{{\bf Q}}
 \nc{\bfY}{{\bf Y}}
 \nc{\bfgb}{{\boldsymbol \gb}}
 \nc{\bfga}{{\boldsymbol \ga}}
 \nc{\bfrho}{{\boldsymbol \rho}}
 \nc{\bfchi}{{\boldsymbol \chi}}
 \nc{\QX}{{\Q\langle \bfX\rangle}}
 \nc{\QY}{{\Q\langle \bfY\rangle}}
 \nc{\CX}{{\C\langle \bfX\rangle}}
 \nc{\CY}{{\C\langle \bfY\rangle}}
 \nc{\QXX}{{\Q\langle\!\langle \bfX\rangle\!\rangle}}
 \nc{\QYY}{{\Q\langle\!\langle \bfY\rangle\!\rangle}}
 \nc{\CXX}{{\C\langle\!\langle \bfX\rangle\!\rangle}}
 \nc{\CYY}{{\C\langle\!\langle \bfY\rangle\!\rangle}}

 \nc{\bbA}{{\mathbb A}}
 \nc{\bbB}{{\mathbb B}}
 \nc{\bbC}{{\mathbb C}}
 \nc{\bbD}{{\mathbb D}}
 \nc{\bbE}{{\mathbb E}}
 \nc{\bbF}{{\mathbb F}}
 \nc{\bbG}{{\mathbb G}}
 \nc{\bbH}{{\mathbb H}}
 \nc{\bbI}{{\mathbb I}}
 \nc{\bbJ}{{\mathbb J}}
 \nc{\bbK}{{\mathbb K}}
 \nc{\bbL}{{\mathbb L}}
 \nc{\bbM}{{\mathbb M}}
 \nc{\bbN}{{\mathbb N}}
 \nc{\bbO}{{\mathbb O}}
 \nc{\bbP}{{\mathbb P}}
 \nc{\bbQ}{{\mathbb Q}}
 \nc{\bbR}{{\mathbb R}}
 \nc{\bbS}{{\mathbb S}}
 \nc{\bbT}{{\mathbb T}}
 \nc{\bbU}{{\mathbb U}}
 \nc{\bbV}{{\mathbb V}}
 \nc{\bbW}{{\mathbb W}}
 \nc{\bbX}{{\mathbb X}}
 \nc{\bbY}{{\mathbb Y}}
 \nc{\bbZ}{{\mathbb Z}}
 \nc{\bba}{{\mathbb a}}
 \nc{\bbb}{{\mathbb b}}
 \nc{\bbc}{{\mathbb c}}
 \nc{\bbd}{{\mathbb d}}
 \nc{\bbe}{{\mathbb e}}
 \nc{\bbf}{{\mathbb f}}
 \nc{\bbg}{{\mathbb g}}
 \nc{\bbh}{{\mathbb h}}
 \nc{\bbi}{{\mathbb i}}
 \nc{\bbk}{{\mathbb k}}
 \nc{\bbl}{{\mathbb l}}
 \nc{\bbm}{{\mathbb m}}
 \nc{\bbn}{{\mathbb n}}
 \nc{\bbo}{{\mathbb o}}
 \nc{\bbp}{{\mathbb p}}
 \nc{\bbq}{{\mathbb q}}
 \nc{\bbr}{{\mathbb r}}
 \nc{\bbs}{{\mathbb s}}
 \nc{\bbt}{{\mathbb t}}
 \nc{\bbu}{{\mathbb u}}
 \nc{\bbv}{{\mathbb v}}
 \nc{\bbw}{{\mathbb w}}
 \nc{\bbx}{{\mathbb x}}
 \nc{\bby}{{\mathbb y}}
 \nc{\bbz}{{\mathbb z}}

 \nc{\calA}{{\mathcal A}}
 \nc{\calB}{{\mathcal B}}
 \nc{\calC}{{\mathcal C}}
 \nc{\calD}{{\mathcal D}}
 \nc{\calE}{{\mathcal E}}
 \nc{\calF}{{\mathcal F}}
 \nc{\calG}{{\mathcal G}}
 \nc{\calH}{{\mathcal H}}
 \nc{\calI}{{\mathcal I}}
 \nc{\calJ}{{\mathcal J}}
 \nc{\calK}{{\mathcal K}}
 \nc{\calL}{{\mathcal L}}
 \nc{\calM}{{\mathcal M}}
 \nc{\calN}{{\mathcal N}}
 \nc{\calO}{{\mathcal O}}
 \nc{\calP}{{\mathcal P}}
 \nc{\calQ}{{\mathcal Q}}
 \nc{\calR}{{\mathcal R}}
 \nc{\calS}{{\mathcal S}}
 \nc{\calT}{{\mathcal T}}
 \nc{\calU}{{\mathcal U}}
 \nc{\calV}{{\mathcal V}}
 \nc{\calW}{{\mathcal W}}
 \nc{\calX}{{\mathcal X}}
 \nc{\calY}{{\mathcal Y}}
 \nc{\calZ}{{\mathcal Z}}
  \nc{\cala}{{\mathcal a}}
 \nc{\calb}{{\mathcal b}}
 \nc{\calc}{{\mathcal c}}
 \nc{\cald}{{\mathcal d}}
 \nc{\cale}{{\mathcal e}}
 \nc{\calf}{{\mathcal f}}
 \nc{\calg}{{\mathcal g}}
 \nc{\calh}{{\mathcal h}}
 \nc{\cali}{{\mathcal i}}
 \nc{\calj}{{\mathcal j}}
 \nc{\calk}{{\mathcal k}}
 \nc{\call}{{\mathcal l}}
 \nc{\calm}{{\mathcal m}}
 \nc{\caln}{{\mathcal n}}
 \nc{\calo}{{\mathcal o}}
 \nc{\calp}{{\mathsf p}}
 \nc{\calq}{{\mathcal q}}
 \nc{\calr}{{\mathcal r}}
 \nc{\cals}{{\mathcal s}}
 \nc{\calt}{{\mathcal t}}
 \nc{\calu}{{\mathcal u}}
 \nc{\calv}{{\mathcal v}}
 \nc{\calw}{{\mathcal w}}
 \nc{\calx}{{\mathcal x}}
 \nc{\caly}{{\mathcal y}}
 \nc{\calz}{{\mathcal z}}

 \nc{\frakA}{{\mathfrak A}}
 \nc{\frakB}{{\mathfrak B}}
 \nc{\frakC}{{\mathfrak C}}
 \nc{\frakD}{{\mathfrak D}}
 \nc{\frakE}{{\mathfrak E}}
 \nc{\frakF}{{\mathfrak F}}
 \nc{\frakG}{{\mathfrak G}}
 \nc{\frakH}{{\mathfrak H}}
 \nc{\frakI}{{\mathfrak I}}
 \nc{\frakJ}{{\mathfrak J}}
 \nc{\frakK}{{\mathfrak K}}
 \nc{\frakL}{{\mathfrak L}}
 \nc{\frakM}{{\mathfrak M}}
 \nc{\frakN}{{\mathfrak N}}
 \nc{\frakO}{{\mathfrak O}}
 \nc{\frakP}{{\mathfrak P}}
 \nc{\frakQ}{{\mathfrak Q}}
 \nc{\frakR}{{\mathfrak R}}
 \nc{\frakS}{{\mathfrak S}}
 \nc{\frakT}{{\mathfrak T}}
 \nc{\frakU}{{\mathfrak U}}
 \nc{\frakV}{{\mathfrak V}}
 \nc{\frakW}{{\mathfrak W}}
 \nc{\frakX}{{\mathfrak X}}
 \nc{\frakY}{{\mathfrak Y}}
 \nc{\frakZ}{{\mathfrak Z}}
 \nc{\fraka}{{\mathfrak a}}
 \nc{\frakb}{{\mathfrak b}}
 \nc{\frakc}{{\mathfrak c}}
 \nc{\frakd}{{\mathfrak d}}
 \nc{\frake}{{\mathfrak e}}
 \nc{\frakf}{{\mathfrak f}}
 \nc{\frakg}{{\mathfrak g}}
 \nc{\frakh}{{\mathfrak h}}
 \nc{\fraki}{{\mathfrak i}}
 \nc{\frakj}{{\mathfrak j}}
 \nc{\frakk}{{\mathfrak k}}
 \nc{\frakl}{{\mathfrak l}}
 \nc{\frakm}{{\mathfrak m}}
 \nc{\frakn}{{\mathfrak n}}
 \nc{\frako}{{\mathfrak o}}
 \nc{\frakp}{{\mathfrak p}}
 \nc{\frakq}{{\mathfrak q}}
 \nc{\frakr}{{\mathfrak r}}
 \nc{\fraks}{{\mathfrak s}}
 \nc{\frakt}{{\mathfrak t}}
 \nc{\fraku}{{\mathfrak u}}
 \nc{\frakv}{{\mathfrak v}}
 \nc{\frakw}{{\mathfrak w}}
 \nc{\frakx}{{\mathfrak x}}
 \nc{\fraky}{{\mathfrak y}}
 \nc{\frakz}{{\mathfrak z}}

 \nc{\sha}{{\mbox{\cyr x}}}

\begin{document}

\title[Colored Tornheim's double series]
{A Note on colored Tornheim's double series}

\author{Jianqiang Zhao}

\maketitle

\begin{center}
Department of Mathematics, Eckerd College, St. Petersburg, FL 33711, USA\\
  Max-Planck Institut f\"ur Mathematik, Vivatsgasse 7, 53111 Bonn, Germany
\end{center}

\medskip
% \date{\today}
% \begin{abstract}
\noindent{\small {\bf Abstract.}
In this short note, we provide an explicit formula to compute every colored
double Tornheim's series by using double polylogarithm values at
roots of unity. When the colors are given by $\pm 1$ our formula is
different from that of Tsumura \cite{Ts1} even though numerical data
confirm both are correct in almost all the cases. This agreement can
also be checked rigorously by using regularized double shuffle relations
of the alternating double zeta values in weights less than eight.

%\keywords{Mordell-Tornheim $L$-values, colored
%Mordell-Tornheim zeta values,
%multiple $L$-values, colored multiple zeta values.}
%\end{abstract}

%\tableofcontents
%\interdisplaylinepenalty=500

\vskip0.7cm

Recently, by using analytic method Tsumura \cite{Ts1,Ts2} studied
the alternating analogues of Tornheim's double series. These
series are the special cases of the so called colored
Mordell-Tornheim zeta values defined in \cite{ZZ}. In depth two,
it has the following form: for any $N$th roots of unity $\ga$ and $\gb$
\begin{equation}\label{equ:claim}
 \zeta_\MT(p,q,r;\ga,\gb ):= \sum_{m,n=1}^\infty
   \frac{ \ga^n\gb^{m+n}}
    {m^p n^q(m+n)^r}
\end{equation}
where $p,q,r\in \Z_{\ge 0}$ such that $p+r,q+r>1$
and $p+q+r>2$. Tsumura's main result gives a formula for \eqref{equ:claim}
when $N=2$, $p,q,r$ are positive and the weight $p+q+r$ is odd.

In \cite{ZZ} the author and Zhou showed that every colored
Mordell-Tornheim zeta values is a $\Q$-linear combination of
colored multiple zeta values (i.e. multiple polylogarithm values
at roots of unity) of the same weight and same depth
(see Thm.~3.2 of loc. cit.) although no explicit formula is given
because the proof there depends on an induction process. Nonetheless,
in small depths it is possible to derive such explicit formulas.
For example, in depth two we have
\begin{thm}\label{thm:main}
Let $p,q,r\in \Z_{\ge 0}$ such that $p+q>0,p+r,q+r>1$
and $p+q+r>2$. Let $\ga,\gb$ be two $N$th roots of unity.
For any positive integers $s$ and $t$ we define
the double polylogarithm value
$Li_{s,t}(\gb,\ga)=\sum_{m>n\ge 1}\gb^m\ga^n/(m^s n^t)$. Then
\begin{equation}\label{equ:thm}
 \zeta_\MT(p,q,r;\ga,\gb )=
\sum_{a=0}^{p-1}  {q+a-1\choose a}Li_{r+q+a,p-a}(\ga\gb, \ga^{-1})
+\sum_{b=0}^{q-1} {p+b-1\choose b}Li_{r+p+b,q-b}(\gb,\ga ).
\end{equation}
\end{thm}
\begin{proof}
We have the following well-known combinatorial identity
(for e.g., see \cite[p.~48]{Nie}):
$$ \frac{1}{x^py^q} =
\sum_{a=0}^{p-1}  {q+a-1\choose a}\frac{1}{x^{p-a}(x+y)^{q+a}}
+\sum_{b=0}^{q-1} {p+b-1\choose b}\frac{1}{y^{q-b}(x+y)^{p+b}}$$
for any two positive integers $p$ and $q$ and any two
real numbers $x$ and $y$ such that $x+y\ne 0$.
This immediately yields the decomposition
\begin{equation*}
 \zeta_\MT(p,q,r;\ga,\gb )=
 \sum_{a=0}^{p-1}  {q+a-1\choose a}\frac{\ga^n\gb^{m+n}}{m^{p-a}(m+n)^{r+q+a}}
+\sum_{b=0}^{q-1} {p+b-1\choose b}\frac{\ga^n\gb^{m+n}}{n^{q-b}(m+n)^{r+p+b}}
\end{equation*}
which gives \eqref{equ:thm}, as desired.
\end{proof}
In \cite{Ts2} Tsumura defines
\begin{equation*}
 R(p,q,r):=\sum_{m,n=1}^\infty   \frac{(-1)^{n}}{m^p n^q(m+n)^r},\qquad
 S(p,q,r):=\sum_{m,n=1}^\infty   \frac{(-1)^{m+n}}{m^p n^q(m+n)^r},
\end{equation*}
which are two special cases of $\zeta_\MT(p,q,r;\ga,\gb )$ when $N=2$.
Define
$$\zeta(\ol{p},\ol{q} ):=\sum_{m>n\ge 1}  \frac{(-1)^{m+n}}{m^p n^q},\quad
\zeta(\ol{p},q ):=\sum_{m>n\ge 1}  \frac{(-1)^{m}}{m^p n^q},\quad
\zeta(p,\ol{q} ):=\sum_{m>n\ge 1}  \frac{(-1)^n}{m^p n^q}.$$
Then the next corollary follows from Theorem \ref{thm:main} at once.
\begin{cor}\label{cor:main}
Let $p,q,r\in \Z_{\ge 0}$  such that $p+q>0,p+r,q+r>1$
and $p+q+r>2$. Then
\begin{align*}
 R(p,q,r)=&
\sum_{a=0}^{p-1}  {q+a-1\choose a}\zeta(\ol{r+q+a},\ol{p-a} )
+\sum_{b=0}^{q-1} {p+b-1\choose b}\zeta(r+p+b,\ol{q-b}),\\
S(p,q,r)=&
\sum_{a=0}^{p-1}  {q+a-1\choose a}\zeta(\ol{r+q+a},p-a)
+\sum_{b=0}^{q-1} {p+b-1\choose b}\zeta(\ol{r+p+b},q-b).
\end{align*}
\end{cor}
For example, we have
\begin{align*}
R(1,1,3)=&\gz(\ol4,\ol1)+\gz(4,\ol1),   \quad
R(1,2,2)= \gz(\ol4,\ol1)+\gz(3,\ol2)+\gz(4,\ol1),\quad
R(1,1,5)= \gz(\ol6,\ol1)+\gz(6,\ol1),\\
R(2,1,2)=&\gz(\ol3,\ol2)+\gz(\ol4,\ol1)+\gz(4,\ol1),   \quad
R(2,3,2)=\gz(\ol5,\ol2)+3\gz(\ol6,\ol1)+\gz(4,\ol3)+2\gz(5,\ol2)+3\gz(6,\ol1),\\  R(1,2,4)=&\gz(\ol6,\ol1)+\gz(5,\ol2)+\gz(6,\ol1),\quad
R(1,3,3)= \gz(\ol6,\ol1)+\gz(4,\ol3)+\gz(5,\ol2)+\gz(6,\ol1),\\  R(2,1,4)=&\gz(\ol5,\ol2)+\gz(\ol6,\ol1)+\gz(6,\ol1),\quad
R(1,4,2)= \gz(\ol6,\ol1)+\gz(3,\ol4)+\gz(4,\ol3)+\gz(5,\ol2)+\gz(6,\ol1),\\  R(2,2,3)=&\gz(\ol5,\ol2)+2\gz(\ol6,\ol1)+\gz(5,\ol2)+2\gz(6,\ol1), \quad R(3,1,3)=\gz(\ol4,\ol3)+\gz(\ol5,\ol2)+\gz(\ol6,\ol1)+\gz(6,\ol1),\\  R(3,2,2)=&\gz(\ol4,\ol3)+2\gz(\ol5,\ol2)+3\gz(\ol6,\ol1)+\gz(5,\ol2)+3\gz(6,\ol1),\\  R(4,1,2)=&\gz(\ol3,\ol4)+\gz(\ol4,\ol3)+\gz(\ol5,\ol2)+\gz(\ol6,\ol1)+\gz(6,\ol1).
\end{align*}
Recently Bl\"umlein, Broadhurst and Vermaseren \cite{BBV} have
built tables of relations for (alternating) multiple zeta values
and so one can rigorously check if the above data agree
with those in \cite{Ts2} or not.
When the weight is $\le 7$ these relations can be produced by
using regularized double shuffle relations
(see \cite{Rac} or \cite{Zpolrel}).
We also have verified this agreement numerically by EZface \cite{EZface}
except that $R(2,1,2)=-.2402184755\cdots$ by our formula while
$R(2,1,2)=\frac{45}{16}\gz(5)-\frac14\pi^2 \gz(3)=-.0495972141\cdots$
by the second line on \cite[p.~90]{Ts2}. With Maple it is easy to compute
this value directly by the series definition and see that our value is correct.
In fact, by using regularized double shuffle relations we get
$R(2,1,2)=\gz(\ol3,\ol2)+\gz(\ol4,\ol1)+\gz(4,\ol1)=
\frac{107}{32}\gz(5)-\frac{5}{16}\pi^2\gz(3).$
\medskip

\noindent
\textbf{Acknowledgement.} The author wishes to thank Vermaseren
for verifying the identities involving alternating double zeta values
in this note. He also thanks
Max-Planck-Institut f\"ur Mathematik for providing financial support
during his sabbatical leave when this work was done.

\end{document}